\documentclass[11pt]{article}
\usepackage{fullpage, amsmath}
\usepackage{amsthm,amsfonts,url,amssymb}
\usepackage{mathrsfs}
\usepackage{amssymb,amscd}
\usepackage{color}
\usepackage{pictexwd,dcpic}
\usepackage{extarrows}
\usepackage[colorlinks,linkcolor=red,anchorcolor=green,citecolor=blue]{hyperref}
\usepackage{lipsum}

\newtheorem{theorem}{Theorem}[section]
\newtheorem{corollary}[theorem]{Corollary}
\newtheorem{lemma}[theorem]{Lemma}

\newtheorem{remark}[theorem]{Remark}

\newtheorem{conjecture}[theorem]{Conjecture}

\newcommand{\hooklongrightarrow}{\lhook\joinrel\longrightarrow}

\newcommand{\ra}{\rightarrow}
\newcommand{\lra}{\longrightarrow}

\newcommand{\ttr}{\texttt r}

\newcommand{\bC}{\mathbb C}

\newcommand{\Q}{\mathbb Q}

\newcommand{\Z}{\mathbb Z}

\newcommand{\bP}{\mathbb P}

\newcommand{\cL}{\mathcal L}

\newcommand{\cR}{\mathcal R}

\newcommand{\cC}{\mathcal C}

\newcommand{\cD}{\mathcal D}

\newcommand{\cE}{\mathcal E}

\newcommand{\ub}{\mathfrak b}

\newcommand{\sW}{\mathscr W}

\DeclareMathOperator{\la}{\mathrm la}

\DeclareMathOperator{\gl}{\mathfrak gl}

\DeclareMathOperator{\GL}{\mathrm GL}

\DeclareMathOperator{\Fil}{\mathrm Fil}
\DeclareMathOperator{\Res}{\mathrm Res}

\DeclareMathOperator{\Sym}{\mathrm Sym}
\DeclareMathOperator{\Gal}{\mathrm Gal}
\DeclareMathOperator{\Hom}{\mathrm Hom}

\DeclareMathOperator{\End}{\mathrm End}

\DeclareMathOperator{\rig}{\mathrm rig}

\DeclareMathOperator{\dR}{\mathrm dR}

\DeclareMathOperator{\unr}{\mathrm unr}

\DeclareMathOperator{\Ker}{\mathrm Ker}
\DeclareMathOperator{\pr}{\mathrm pr}

\DeclareMathOperator{\Ext}{\mathrm Ext}

\DeclareMathOperator{\Ima}{\mathrm Im}

\DeclareMathOperator{\dett}{\mathrm det}
\DeclareMathOperator{\alg}{\mathrm alg}

\DeclareMathOperator{\Ad}{\mathrm Ad}

\DeclareMathOperator{\ab}{\mathrm ab}

\DeclareMathOperator{\rec}{\mathrm rec}

\DeclareMathOperator{\pst}{pst}
\DeclareMathOperator{\DF}{DF}

\begin{document}
	\title{Locally analytic $\Ext^1$ for  $\GL_2(\Q_p)$ in de Rham non trianguline case}
	\author{Yiwen Ding}
	\date{}
	\maketitle
	\begin{abstract}We prove Breuil's conjecture on locally analytic $\Ext^1$ for $\GL_2(\Q_p)$ in de Rham non-trianguline case.
	\end{abstract}
\section{Introduction}\label{secInt}
Let $E$ be a finite extension of $\Q_p$, $\cR_E$ be the Robba ring with $E$-coefficients. The (locally analytic) $p$-adic local Langlands correspondence for $\GL_2(\Q_p)$ associates to a $(\varphi, \Gamma)$-module $D$ of rank $2$ over $\cR_E$ a locally analytic representation $\pi(D)$ of $\GL_2(\Q_p)$ over $E$ (see for example \cite[\S~0.1]{Colm16}). The representation $\pi(D)$ determines $D$ (and vice versa). Indeed, when $D$ is trianguline, this follows from the explicit structure of $\pi(D)$ and $D$. When $D$ is not trianguline, one can reduce to the case where $D$ is \'etale hence isomorphic to $D_{\rig}(\rho)$ for a certain $2$-dimensional representation $\rho$ of  the absolute Galois group $\Gal_{\Q_p}$ over $E$. In this case, by \cite[Thm. 0.2]{CD}, the universal completion of $\pi(D)$ is exactly the  Banach representation of $\GL_2(\Q_p)$ associated to $\rho$, which determines $\rho$ (hence $D$) via Colmez's Montreal functor (see \cite[Thm. 0.17 (iii)]{Colm10a}).

 The $p$-adic local Langlands correspondence is compatible with (and refines) the classical local Langlands correspondence. We recall the feature in more details. Suppose that  $D$ is de Rham of  Hodge-Tate weights $(0,k)$ with $k\geq 1$ (where we use the convention that the Hodge-Tate weight of the cyclotomic character is $1$). We can associate to $D$ a smooth $\GL_2(\Q_p)$-representation in the following way:
 \begin{equation*}
\underbrace{D \longleftrightarrow  D_{\pst}(D) \leadsto \DF}_{\text{$p$-adic Hodge theory}} \longleftrightarrow \underbrace{\ttr \longleftrightarrow \pi_{\infty}(\ttr)}_{\text{local Langlands}}
 \end{equation*}
 where 
 \begin{itemize}
 	\item $D_{\pst}(D)$ is the filtered $(\varphi, N, \Gal(L/\Q_p))$-module associated to $D$ (cf. \cite[Thm. A]{Ber08a}), where $L$ is a certain finite extension of $\Q_p$,
 	\item $\DF$ is the underlying Deligne-Fontaine module (i.e. $(\varphi, N, \Gal(L/\Q_p))$-module) of $D_{\pst}(D)$ (by forgetting the Hodge filtration),
 	\item  $\ttr$ is the $2$-dimensional Weil-Deligne representation associated to $\DF$ as in \cite[\S~4]{BS07},
 	\item $\pi_{\infty}(\ttr):=\rec^{-1}(\ttr)$ is the smooth $\GL_2(\Q_p)$-representation associated to $\ttr$ via the classical local Langlands correspondence (normalized as in \cite{HT}, in particular, the central character $\omega_{\pi_{\infty}(\ttr)}$ is $\wedge^2 \ttr \otimes_E \unr(p)$, where we view the one-dimensional Weil representation $\wedge^2 \ttr$ as a character of $\Q_p^{\times}$ via $W_{\Q_p}^{\ab}\cong \Q_p^{\times}$, normalized by sending geometric Frobenius to uniformizers, and where $\unr(p)$ is the unramified character of $\Q_p^{\times}$ sending uniformizers to $p$).  
 \end{itemize}
Put $\pi_{\alg}(\ttr, k):=\Sym^{k-1} E^2 \otimes_E \pi_{\infty}(\ttr)$, which is a locally algebraic representation of $\GL_2(\Q_p)$ (for the diagonal action). Then there is  a natural injection (\cite[Thm. 3.3.22]{Em4}):
 \begin{equation*}
 	\pi_{\alg}(\ttr,k) \hooklongrightarrow \pi(D).
 \end{equation*}
It turns out that the quotient $\pi_c(\ttr,k):=\pi(D)/\pi_{\alg}(\ttr,k)$ (that is  a locally analytic representation of $\GL_2(\Q_p)$ as well) also depends only on and determines $\{\ttr, k\}$ (see \cite[\S~0.2]{Colm18}). One may view the correspondence $\pi_c(\ttr,k) \leftrightarrow \{\ttr,k\}$ as a  local Langlands correspondence  for the simple reflection in the Weyl group $\sW\cong S_2$ of $\GL_2$ (see \cite[Remark 5.3.2 (iv)]{BD2} for related discussions).

We let $\Delta$ be the $p$-adic differential equation associated to $D$, i.e. the $(\varphi, \Gamma)$-module associated to $\DF$ equipped with the trivial Hodge filtration  via \cite[Thm. A]{Ber08a}. By \textit{loc. cit.},  the category of $p$-adic differential equations is equivalent to the category of Deligne-Fontaine modules (that is  equivalent to the category of Weil-Deligne representations). We have natural isomorphisms $D_{\pst}(\Delta)\xrightarrow{\sim} \DF$ (as Deligne-Fontaine module), and $D_{\dR}(\Delta)\xrightarrow{\sim} D_{\dR}(D)$ (as $E$-vector space). 
The Hodge filtration on $D_{\dR}(D)$ has the following form
\begin{equation}\label{HFil}\Fil^i D_{\dR}(D)=\begin{cases}
D_{\dR}(\Delta) &i\leq -k \\
\cL(D) & -k<i \leq 0\\
0 & i>0
\end{cases},
\end{equation} where $\cL(D)$ is a certain $E$-line in $D_{\dR}(\Delta)$. By \cite[Thm. A]{Ber08a}, $D$ is equivalent to the data $\{\Delta, k, \cL(D)\}$  (or equivalently $\{\ttr, k, \cL(D)\}$).  And we see when we pass from $D$ to $\{\ttr,k\}$, we lose exactly the information on $\cL(D)$. To make the notation more consistent, we write $\pi_{\alg}(\Delta,k):=\pi_{\alg}(\ttr,k)$, and $\pi_c(\Delta,k):=\pi_c(\ttr,k)$. As the whole locally analytic $\GL_2(\Q_p)$-representation $\pi(D)$ can determine $D$ while the constituents $\pi_{\alg}(\Delta,k)$, $\pi_c(\Delta,k)$ only determine $\{\Delta,k\}$, this suggests the information on $\cL(D)$ should be contained in the corresponding extension class  (see \cite[\S~2.1]{Br16} for the definition of  $\Ext^1_{\GL_2(\Q_p)}$) 
$$[\pi(D)]\in  \Ext^1_{\GL_2(\Q_p)}\big(\pi_c(\Delta,k), \pi_{\alg}(\Delta,k)\big).$$
 In \cite{Br16}, Breuil formulated the following conjecture  in this direction (see  \cite[Conj. 1.1]{Br16} for   general $\GL_n$-case):
\begin{conjecture}\label{conjEXT}
	There is a natural $E$-linear bijection
	\begin{equation}\label{EXT1}
\Ext^1_{\GL_2(\Q_p)}\big(\pi_c(\Delta,k), \pi_{\alg}(\Delta,k)\big) \xlongrightarrow{\sim} D_{\dR}(\Delta)
	\end{equation}
	such that for any de Rham  $(\varphi, \Gamma)$-module $D$ of rank $2$ over $\cR_E$ of Hodge-Tate weights $(0,k)$ with the  associated $p$-adic differential equation isomorphic to $\Delta$, the map sends  the $E$-line $E[\pi(D)]$ to $\cL(D)$.
\end{conjecture}
The conjecture was proved in the trianguline case (or equivalently, when $\Delta$ (or equivalently $\ttr$) is reducible) in \cite[\S~3.1]{Br16}. The proof relied on a direct calculation of $\Ext^1_{\GL_2(\Q_p)}\big(\pi_c(\Delta,k), \pi_{\alg}(\Delta,k)\big)$. Indeed, when $D$ (or equivalently $\Delta$) is trianguline, the irreducible constituents of $\pi(D)$ are among those that appear in locally analytic principal series, so such a  calculation can be carried out. In this note, we prove the conjecture in de Rham non-trianguline case hence complete all cases. In fact, we prove a refined version of the conjecture given in \cite[Conj. 5.3.1]{BD2} (see Corollary \ref{corcE} and Theorem \ref{thmext}), which describes the bijection in Conjecture  \ref{conjEXT} in a functorial way.  

\begin{remark}
	When $\Delta$ is de Rham non-trianguline, by \cite[Thm. 0.6]{Colm18}, there is an injective $E$-linear map
	\begin{equation}\label{ext01}D_{\dR}(\Delta) \hooklongrightarrow \Ext^1_{\GL_2(\Q_p)}\big(\pi_c(\Delta,k), \pi_{\alg}(\Delta,k)\big)
	\end{equation} satisfying the same property as (the inverse) of  (\ref{EXT1}). Hence $\dim_E \Ext^1_{\GL_2(\Q_p)}\big(\pi_c(\Delta,k), \pi_{\alg}(\Delta,k)\big)\geq 2$, and one may prove the conjecture by showing the equality holds. Indeed, by \cite[Thm. 0.6 (iii)]{Colm18}, one has an extension (which is the universal extension \textit{a posteriori}, where $\pi(\Delta,k)$ is the representation $\Pi(M,k)$ of \textit{loc. cit})
	\begin{equation}\label{univ}
	0 \ra \pi_{\alg}(\Delta,k)\otimes_E D_{\dR}(\Delta) \ra \pi(\Delta,k) \ra \pi_c(\Delta,k) \ra 0,
	\end{equation}
	satisfying that for  any de Rham  $(\varphi, \Gamma)$-module $D$ of rank $2$ over $\cR_E$ of Hodge-Tate weights $(0,k)$ with the  associated $p$-adic differential equation isomorphic to $\Delta$, $\pi(D)\cong \pi(\Delta,k)/(\pi_{\alg}(\Delta,k) \otimes_E \cL(D))$. The extension class $[\pi(\Delta,k)]$ induces via the natural cup-product
	\begin{equation*}
	\Ext^1_{\GL_2(\Q_p)}\big(\pi_c(\Delta,k), \pi_{\alg}(\Delta,k)\otimes_E D_{\dR}(\Delta)\big) \times D_{\dR}(\Delta)^{\vee} \ra \Ext^1_{\GL_2(\Q_p)}\big(\pi_c(\Delta,k), \pi_{\alg}(\Delta,k)\big)
	\end{equation*}
	an $E$-linear map
	\begin{equation}\label{ext00}D_{\dR}(\Delta)^{\vee} \lra \Ext^1_{\GL_2(\Q_p)}\big(\pi_c(\Delta,k), \pi_{\alg}(\Delta,k)\big)
	\end{equation}
	sending an $E$-line $\cL(D)^{\perp}=(D_{\dR}(\Delta)/\cL(D))^{\vee} \hookrightarrow D_{\dR}(\Delta)^{\vee}$ to $E [\pi(D)]$. Note that (\ref{ext00}) is injective, as for different $E$-lines $\cL(D_1)\neq \cL(D_2)$, we have $D_1\ncong D_2$ hence $\pi(D_1)\ncong \pi(D_2)$. Let $e_1, e_2$ be a basis of $D_{\dR}(\Delta)$, and $e_i^*\in D_{\dR}(\Delta)^{\vee}$ such that $e_i^*(e_j)=\begin{cases}1 & i=j \\ 0 & i\neq j\end{cases}$. We see the $E$-linear bijective map $D_{\dR}(\Delta) \xrightarrow{\sim} D_{\dR}(\Delta)^{\vee}$ given by  $e_1 \mapsto e_2^*$, $e_2\mapsto -e_1^*$ sends each $E$-line $\cL$ to $\cL^{\perp}$. This bijection pre-composed with (\ref{ext00}) gives then the injection in (\ref{ext01}). Finally, we remark that  by \cite[Thm. 1.4]{DLB}, the  universal extension (\ref{univ}) can be realized in the de Rham complex of the coverings of Drinfeld's upper half-plane.
\end{remark}

\section{Main results}
\label{secMain}
Before stating our main results, we quickly introduce some more notation.  For $r\in \Q_{>0}$, let $\cR_E^r$ be the  Fr\'echet space of $E$-coefficient rigid analytic functions on the annulus $p^{-\frac{1}{r}}\leq |\cdot|<1$ where $|\cdot|$ is the norm on $\bC_p$ normalized such that $|p|=p^{-1}$.
We have $\cR_E\cong \varinjlim_r \cR_E^r$. Let $\cR_E^+$ be the Fr\'echet space of $E$-coefficient rigid analytic functions on  the open unit disk $|\cdot|<1$:
\begin{equation*}
	\cR_E^+=\big\{\sum_{i=0}^{+\infty} a_i X^i\ |\ a_i\in E \text{ for all $i$, and } |a_i| r^i \ra 0, i \ra +\infty \text{ for all $0\leq r <1$}\big\}.
\end{equation*}We have  $\cR_E^+\hookrightarrow \cR_E^r$ for all $r$. The Robba ring $\cR_E$ is equipped with a natural (standard) action of $\Gamma\cong \Z_p^{\times}$ and operators $\varphi$ and $\psi$. Recall that the $\Gamma$-action sends $\cR_E^+$ (resp. $\cR_E^r$) to $\cR_E^+$ (resp. $\cR_E^r$), and the $\psi$-operator sends $\cR_E^+$ (resp. $\cR_E^r$) to $\cR_E^+$ (resp. $\cR_E^r$ for $r\in \Q_{>p-1}$). Let $D$ be a generalized $(\varphi, \Gamma)$-module over $\cR_E$ (cf. \cite[\S~4.1]{Liu07}, noting $D$ is allowed to  have $t$-torsions, $t=\log (1+X)$). Recall (see \cite[Remark 2.2.2]{BD2} and the discussion above it) there exist $r\in \Q_{>p-1}$, and a generalized $(\varphi, \Gamma)$-module $D_r$ over $\cR_E^r$  (cf. \textit{loc. cit.}) such that $f_r: D_r \otimes_{\cR_E^r} \cR_E \xrightarrow{\sim} D$. In fact, such $\{r, D_r, f_r\}$ form a filtered category $I(D)$, and $\varinjlim_{(r, f_r, D_r)\in I(D)} D_r\xrightarrow{\sim} D$ (see the discussion above \cite[Remark 2.2.4]{BD2}). 

Recall (see for example \cite[\S~1.1.2]{Colm18}) that $\cR_E^+$ is naturally isomorphic to the locally analytic distribution algebra $\cD(\Z_p,E)=\cC^{\la}(\Z_p,E)^{\vee}$ of $\Z_p$. Under this isomorphism, the operators $\varphi$, $\psi$, and $\gamma \in \Gamma$ can be described as follows: for $\mu\in \cD(\Z_p,E)$, $f\in \cC^{\la}(\Z_p,E)$,
\begin{equation*}
\varphi(\mu)(f)=\mu([x \mapsto f(px)]), \ \psi(\mu)(f)=\mu([x\mapsto f(\frac{x}{p})]),
\end{equation*}
\begin{equation*}\gamma (\mu)(f)=\mu([x\mapsto f(\gamma x)]).
\end{equation*}
The element $t\in \cR_E^+$ (resp. $X$) corresponds to the distribution $f \mapsto f'(0)$ (resp. $f \mapsto f(1)-f(0)$).

Let $\pi$ be an admissible locally analytic representation of $\GL_2(\Q_p)$ over $E$.  The continuous dual $\pi^{\vee}$ of $\pi$ (equipped with the strong topology) is then a Fr\'echet space over $E$.  The action of $N(\Z_p)=\begin{pmatrix} 1 & \Z_p \\ 0 & 1\end{pmatrix}$ on $\pi$ induces a (separately-continuous)  $\cR_E^+$-module structure  on $\pi^{\vee}$. Note that $t\in \cR_E^+$ acts on $\pi^{\vee}$ via the element   $u_+:=\begin{pmatrix}
	0 & 1  \\ 0 & 0 
\end{pmatrix}$ in the Lie algebra $\gl_2$ of $\GL_2(\Q_p)$, and we identify $u_+$ and $t$ frequently without further mention. Moreover, the $\cR_E^+$-module $\pi^{\vee}$ is equipped with an operator $\psi$  given by the action of $\begin{pmatrix} \frac{1}{p} & 0 \\ 0 & 1\end{pmatrix}$, and with an action of $\Gamma\cong \Z_p^{\times}$ given by the action of $\begin{pmatrix} \Z_p^{\times}  & 0 \\ 0 & 1\end{pmatrix}$ satisfying
\begin{equation*}
\psi(\varphi(x) v)=x\psi(v), \gamma(xv)=\gamma(x) \gamma(v)
\end{equation*}
for $x\in \cR_E^+$, $v\in \pi^{\vee}$ and $\gamma \in \Gamma$. Recall in \cite[\S~2.3]{BD2} (see in particular \cite[Ex. 2.3.3]{BD2}), we associated to $\pi$ a (covariant) functor  $F(\pi)$ from the category of generalized $(\varphi, \Gamma)$-modules to the category of $E$-vector spaces:
\begin{equation*}
F(\pi)(D)=\varinjlim_{(r,f_r,D_r)\in I(D)} \Hom_{(\psi, \Gamma)}(\pi^{\vee}, D_r),
\end{equation*}
where $\Hom_{(\psi, \Gamma)}$ consists of continuous $\cR_E^+$-linear morphisms that are $(\psi, \Gamma)$-equivariant. Note that if $D$ has no $t$-torsion, then by \cite[Cor. 8.9]{NFA}, we have $F(\pi)(D)=\Hom_{(\psi, \Gamma)}(\pi^{\vee}, D)$ (where $D$ is equipped with the inductive limit topology). Let $M(\Q_p):=\bigg\{\begin{pmatrix}
a & b \\ 0 & 1
\end{pmatrix}\ |\ a\in \Q_p^{\times}, \ b\in \Q_p\bigg\}$. By definition, the functor $F(\pi)$ only depends on $\pi|_{M(\Q_p)}$.

Our first result is on the representability of $F(\pi)$ in de Rham non-trianguline case. Namely, let $\Delta$ be an irreducible  $(\varphi, \Gamma)$-module free of rank $2$ over $\cR_E$, de Rham of constant Hodge-Tate weight $0$. Let $\pi(\Delta)$ be the locally analytic representation associated to $\Delta$ (cf. \cite[\S~2.1]{Colm18}) normalized such that the central character $\omega_{\Delta}$ of $\pi(\Delta)$ satisfies $\cR_E(\omega_{\Delta})\cong \wedge^2 \Delta \otimes_{\cR_E} \cR_E(\varepsilon^{-1})$, where $\varepsilon=z |z|^{-1}: \Q_p^{\times} \ra E^{\times}$ and for a continuous character $\delta: \Q_p^{\times} \ra E^{\times}$, we denote by $\cR_E(\delta)$ the associated rank one $(\varphi, \Gamma)$-module.  Let $\check{\Delta}:=\Delta^{\vee} \otimes_{\cR_E} \cR_E(\varepsilon) \cong \Delta \otimes_{\cR_E} \cR_E(\omega_{\Delta}^{-1})$ be the Cartier dual of $\Delta$.
\begin{theorem}\label{thmDel}
	The functor $F(\pi(\Delta))$ is representable by $\check{\Delta}$, i.e. for any generalized $(\varphi,\Gamma)$-module $D$, $F(\pi(\Delta))(D)=\Hom_{(\varphi, \Gamma)}(\check{\Delta}, D)$. 
\end{theorem}
\begin{remark}
 The same statement in the trianguline case was obtained  in \cite[Thm. 5.4.2 (i)]{BD2} (see Step 2 \& 3 of the proof), where a key ingredient is the representability of $F(\pi)$ for locally analytic principal series $\pi$. While, our proof of Theorem \ref{thmDel}  is based on Colmez's results in \cite{Colm18} and the representability of $F(\pi)$ for locally algebraic representations $\pi$. 
\end{remark}
For $k\in \Z$, let $\pi(\Delta,k)$ be the locally analytic representation $\Pi(M,k)$ in \cite[Thm. 0.8 (iii)]{Colm18} (for $M=\DF$, the irreducible Deligne-Fontaine module associated to $\Delta$). Recall by \textit{loc. cit.}, there exists an isomorphism of topological $E$-vector spaces: $\partial: \pi(\Delta) \ra \pi(\Delta)$ such that the following maps 
\begin{equation*}
g=\begin{pmatrix}
a & b \\ c&d
\end{pmatrix}: \pi(\Delta) \ra \pi(\Delta), \ v \mapsto (-c \partial +a)^k  \begin{pmatrix}
a & b \\ c & d 
\end{pmatrix}(v)
\end{equation*}
define a(nother) locally analytic $\GL_2(\Q_p)$-action and the resulting representation is isomorphic to  $\pi(\Delta,k)$. As $B(\Q_p)$-representation, we have $\pi(\Delta,k)\cong \pi(\Delta) \otimes_E (x^k \otimes 1)$ hence:
\begin{equation*}
F(\pi(\Delta))(D)\cong F(\pi(\Delta,k))(D\otimes_{\cR_E} \cR_E(x^{-k}))
\end{equation*}
for all generalized $(\varphi, \Gamma)$-modules $D$. We then deduce from Theorem 1.1:
\begin{corollary}\label{fpk1}
	The functor $F(\pi(\Delta,k))$ is representable by $t^{-k} \check{\Delta}$.
\end{corollary}

As in \S~\ref{secInt}, let $\pi_{\infty}(\Delta)$ be the smooth representation of $\GL_2(\Q_p)$ associated to $\Delta$, and for $k\in \Z_{\geq 1}$, let  $\pi_{\alg}(\Delta,k):=\Sym^{k-1} E^2 \otimes_E \pi_{\infty}(\Delta)$ and  $\pi_c(\Delta,k):=\pi(\Delta,-k)\otimes_E (x^k \circ \dett)$. Recall by \cite[Thm. 3.3.1]{BD2}, $F(\pi_{\alg}(\Delta,k))$ is representable by $\cR_E(x^{1-k})/t^{k}$. By Theorem  \ref{thmDel} and $\pi_c(\Delta,k)|_{M(\Q_p)}\cong \pi(\Delta)|_{M(\Q_p)}$,  we see $F(\pi_c(\Delta, k))=F(\pi(\Delta))$ is representable by $\check{\Delta}$.  By \cite[Thm. 4.1.5]{BD2},  we then obtain:
\begin{corollary}\label{corcE}
	There exists a natural $E$-linear map
	\begin{equation}\label{eqCE}\cE:
	\Ext^1_{\GL_2(\Q_p)}\big(\pi_c(\Delta,k), \pi_{\alg}(\Delta,k)\big) \lra \Ext^1_{(\varphi, \Gamma)}\big(\cR_E(x^{1-k})/t^k, \check{\Delta}\big)
	\end{equation}
	satisfying 
	that for  $[\pi]\in \Ext^1_{\GL_2(\Q_p)}\big(\pi_c(\Delta,k), \pi_{\alg}(\Delta,k)\big)$, the functor $F(\pi)$ is representable by the extension of class $\cE([\pi])$. 
\end{corollary} 
By \cite[Prop. 5.1.2]{BD2}, there is a natural isomorphism of  $E$-vector spaces: 
\begin{equation}\label{filmaxD}\Ext^1_{(\varphi, \Gamma)}\big(\cR_E(x^{1-k})/t^k, \check{\Delta}\big)\xlongrightarrow{\sim} D_{\dR}(\check{\Delta})
\end{equation}
 satisfying that for each non-split $[D]\in \Ext^1_{(\varphi, \Gamma)}\big(\cR_E(x^{1-k})/t^k, \check{\Delta}\big)$, the map sends the line $E[D]$ to $\cL(D)\hookrightarrow D_{\dR}(D)\cong D_{\dR}(\check{\Delta})$, where $\cL(D)$ is defined in a similar way as in (\ref{HFil}): $\cL(D):=\Fil^{\max} D_{\dR}(D)=\Fil^i D_{\dR}(D)$ for $i=0, \cdots, k-1$ (noting such $D$ has Hodge-Tate weights $(1-k, 1)$). Using the isomorphism $D_{\dR}(\check{\Delta}) \cong D_{\dR}(\Delta)$ (with a shift of the Hodge filtration) induced by $\check{\Delta}\cong \Delta \otimes_{\cR_E} \cR_E(\omega_{\Delta}^{-1})$, the composition of (\ref{eqCE}) and (\ref{filmaxD}) gives a map as in (\ref{EXT1}) (satisfying the properties below (\ref{EXT1})). Conjecture \ref{conjEXT} (in de Rham non-trianguline case) then follows from the following theorem.
\begin{theorem}\label{thmext}
	The map $\cE$ is bijective, in particular, $$\dim_E \Ext^1_{\GL_2(\Q_p)}\big(\pi_c(\Delta,k), \pi_{\alg}(\Delta,k)\big)=2$$ and any non-split extension of $\pi_c(\Delta,k)$ by $\pi_{\alg}(\Delta,k)$ is associated to a $(\varphi, \Gamma)$-module of rank $2$ over $\cR_E$.
\end{theorem}
By an easy variation of the  proof of Theorem \ref{thmext}, we also obtain the following result on locally analytic $\Ext^1$:
\begin{corollary}\label{corext}
Let $\pi_{\infty}$ be a generic irreducible smooth representation of $\GL_2(\Q_p)$ over $E$, and $W$ be an irreducible algebraic representation of $\GL_2(\Q_p)$ over $E$. Then $\Ext^1_{\GL_2(\Q_p)}(\pi_c(\Delta,k), \pi_{\infty} \otimes_E W)\neq 0$ if and only if $\pi_{\infty}\otimes_E W\cong \pi_{\alg}(\Delta,k)$. 
\end{corollary} 
\section{Proofs}
We keep the notation in \S~\ref{secMain}. 
Let $D_r$ be a generalized $(\varphi, \Gamma)$-module over $\cR_E^r$ (cf. \cite[\S~2.2]{BD2}). We call $D_r$ is good if $D_r\cong (\cR_E^r)^{m_1} \oplus \oplus_{i=1}^{m_2} \cR_E^r/t^{s_i}$ as $\cR_E^r$-module for some integers $m_1\geq 0$, $m_2\geq 0$, and $s_i\geq 1$. And if so, we call $m=m_1+m_2$ the rank of $D_r$. Note  for a general $D_r$, $D_r \otimes_{\cR_E^r} \cR_E^{r'}$ is good for $r'\gg r$ (cf. \cite[(22)]{BD2}). 
We begin with a key lemma.
\begin{lemma}\label{keyL}
	Let $D_r$ be a good generalized $(\varphi, \Gamma)$-module over $\cR_E^r$, and  $f\in \Hom_{(\psi, \Gamma)}(\pi(\Delta)^{\vee}, D_r)$. Suppose the induced morphism 
	\begin{equation*}
	\pi(\Delta)^{\vee} \otimes_{\cR_E^+} \cR_E^r \lra D_r
	\end{equation*}
	has dense image.
	Then the rank of $D_r$ is at most $2$.
\end{lemma}
\begin{proof}
	As $\pi(\Delta,1)\cong \pi(\Delta) \otimes (x \otimes 1)$ as $B(\Q_p)$-representation, we have $$\Hom_{(\psi, \Gamma)}\big(\pi(\Delta,1)^{\vee}, D_r \otimes_{\cR_E^r} \cR_E^r(x^{-1})\big)=\Hom_{(\psi, \Gamma)}(\pi(\Delta)^{\vee}, D_r).$$ In particular, $f$ induces a  morphism
	\begin{equation*}
	\pi(\Delta,1)^{\vee} \otimes_{\cR_E^+} \cR_E^r \longrightarrow D_r \otimes_{\cR_E^r} \cR_E^r(x^{-1})=:D_r(x^{-1}),
	\end{equation*}
	which has dense image.
	This morphism further induces a  morphism with dense image:
	\begin{equation*}
\big(\pi(\Delta,1)^{\vee}/u_+ \pi(\Delta,1)\big) \otimes_{\cR_E^+} \cR_E^r\cong 	\pi(\Delta,1)^{\vee}/t \otimes_{\cR_E^+} \cR_E^r \longrightarrow D_r(x^{-1})/t.
	\end{equation*}
Using \cite[Cor. 9.3]{DLB}, we see  $\pi(\Delta,1)^{\vee}/u_+ \pi(\Delta,1)^{\vee}\cong (\pi(\Delta,1)[u_+])^{\vee}$, where $(-)[u_+]$ denotes the subspace annihilated by $u_+$.  By \cite[Lemma 3.24, Thm. 3.31]{Colm18}, $\pi(\Delta,1)[u_+]\subset \pi(\Delta,1)$ is stabilized by $\GL_2(\Q_p)$, and is isomorphic to $\pi_{\alg}(\Delta,1)^{\oplus 2}\cong \pi_{\infty}(\Delta)^{\oplus 2}$ as $\GL_2(\Q_p)$-representation. By \cite[Thm. 3.3.1]{BD2}, the induced  morphism $\pi(\Delta,1)^{\vee}/t \ra D_r(x^{-1})/t \otimes_{\cR_E^r} \cR_E^{r'}$ for $r' \gg r$ factors through $(\cR_E^{r'}/t)^{\oplus 2}$. For such $r'$, we obtain thus a (continuous $\cR_E^{r'}$-linear) morphism with dense image $(\cR_E^{r'}/t)^{\oplus 2} \ra D_r(x^{-1})/t \otimes_{\cR_E^r} \cR_E^{r'}$. As $\cR_E^{r'}$ is B\'ezout (see for example \cite[Prop. 4.12]{Berger}), it is not difficult to see the rank of $D_r(x^{-1})/t \otimes_{\cR_E^r} \cR_E^{r'}$ is at most $2$ ($=$ the rank of $(\cR_E^{r'}/t)^{\oplus 2}$).  Since the rank of $D_r(x^{-1})/t \otimes_{\cR_E^r} \cR_E^{r'}$ over $\cR_E^{r'}/t$ is the same as the rank of $D_r$,  the lemma follows.
\end{proof}
\begin{proof}[Proof of Theorem \ref{thmDel}]
	Recall (e.g. see \cite[\S~2.1]{Colm18}) $\Delta$ extends uniquely to a $\GL_2(\Q_p)$-sheaf over $\bP^1(\Q_p)$ of central character $\omega_{\Delta}$,  and the space $\Delta \boxtimes \bP^1$ of global sections  sit in a $\GL_2(\Q_p)$-equivariant exact sequence
	\begin{equation*}
	0 \ra \pi(\Delta)^{\vee} \otimes_E (\omega_{\Delta} \circ \dett) \ra \Delta \boxtimes \bP^1 \ra \pi(\Delta) \ra 0.
	\end{equation*}
The space of sections of the $\GL_2(\Q_p)$-sheaf on the open set $\Z_p\hookrightarrow \Q_p \hookrightarrow \bP^1(\Q_p)$ is isomorphic to $\Delta$, and the composition $\iota: \pi(\Delta)^{\vee}\otimes_E (\omega_{\Delta} \circ \dett) \hookrightarrow \Delta \boxtimes \bP^1 \xlongrightarrow{\Res_{\Z_p}} \Delta$ is $\cR_E^+$-linear, continuous and $(\psi, \Gamma)$-equivariant. By the same argument as in Step 1 of the proof of \cite[Thm. 5.4.2]{BD2} (noting since $\Delta$ is irreducible, $\Delta$ is \'etale up to twist by characters),  $\iota$ has image in $\Delta_r$ for $r$ sufficiently large (where $\Delta_r$ is a $(\varphi, \Gamma)$-module over $\cR_E^r$ such that $\Delta_r \otimes_{\cR_E^r} \cR_E \cong \Delta$), and induces a surjective morphism $\iota: \pi(\Delta)^{\vee} \otimes_{\cR_E^+} \cR_E^r \twoheadrightarrow \Delta_r$. 
	
	Let $D$ be a generalized $(\varphi, \Gamma)$-module over $\cR_E$, and let $\mu \in F\big(\pi(\Delta) \otimes_E (\omega_{\Delta}^{-1}\circ \dett)\big)(D)$. Let $(r,D_r, f_r)\in I(D)$ such that $\mu \in \Hom_{(\psi, \Gamma)}\big(\pi(\Delta) \otimes_E (\omega_{\Delta}^{-1}\circ \dett), D_r\big)$. It is sufficient to show that, enlarging $r$ if needed, $\mu$ factors through $\iota$. Indeed, if so,  the following map induced by $\iota$ (see \cite[Lemma 2.2.3 (iii), Remark 2.3.1 (iv)]{BD2}):
	\begin{equation*}
	\Hom_{(\varphi,\Gamma)}(\Delta, D) \lra F\big(\pi(\Delta) \otimes_E (\omega_{\Delta}^{-1} \circ \dett)\big)(D)
	\end{equation*}
	is surjective hence bijective (as $\iota$ is surjective after tensoring the source by $\cR_E^r$). The theorem then follows using $\check{\Delta}\cong \Delta \otimes_{\cR_E} \cR_E(\omega_{\Delta}^{-1})$. 
	
	Replacing $r$ by $r'\gg r$ (and $D_r$ by $D_r \otimes_{\cR_E^r} \cR_E^{r'}$), we  can and do assume that $\iota$ factors through $\Delta_r$, and $D_r$ is good. 	Consider 
	\begin{equation*}
\tilde{\mu}:	\pi(\Delta)^{\vee} \otimes_E (\omega_{\Delta} \circ \dett) \xlongrightarrow{(\iota, \mu)} \Delta_r \oplus D_r.
	\end{equation*}
	Denote by $M_r$ the closed $\cR_E^r$-submodule of $\Delta_r \oplus D_r$ generated by $\Ima(\tilde{\mu})$. As $\tilde{\mu}$ is $(\psi,\Gamma)$-equivariant, we see $M_r\subset \Delta_r \oplus D_r$ is stabilized by $\psi$ and  $\Gamma$. By the discussion in the end of \cite[\S~2.2]{BD2} (see in particular \cite[(24)]{BD2}), $M_r$ is stabilized by $\varphi$ and $\Gamma$, hence is a generalized $(\varphi, \Gamma)$-module over $\cR_E^r$. For $r'\geq r$, $M_{r'}:=M_r \otimes_{\cR_E^r} \cR_E^{r'}$ is the closed $\cR_E^{r'}$-submodule of $\Delta_{r'} \oplus D_{r'}:=(\Delta_r \otimes_{\cR_E^r} \cR_E^{r'}) \oplus (D_r \otimes_{\cR_E^r} \cR_E^{r'})$ generated by the image of $\tilde{\mu}: \pi(\Delta)^{\vee} \otimes_E (\omega_{\Delta} \circ \dett) \xlongrightarrow{(\iota, \mu)} \Delta_{r'} \oplus D_{r'}$. Let $r'$ be sufficiently large such that $M_{r'}$ is good. Then by Lemma \ref{keyL}, the rank $M_{r'}$ is at most $2$. 
	
The following composition 
\begin{equation*}(\pi(\Delta)^{\vee} \otimes_E (\omega_{\Delta} \circ \dett))\otimes_{\cR_E^+} \cR_E^{r'} \ra M_{r'} \hookrightarrow \Delta_{r'} \oplus D_{r'} \xrightarrow{\pr_1} \Delta_{r'}
\end{equation*} is equal to $\iota$ hence surjective. We see the induced morphism $\kappa: M_{r'}\ra  \Delta_{r'}$ is surjective. It is clear that $\kappa$ is continuous $\cR_E^{r'}$-linear and $(\psi, \Gamma)$-equivariant. By \cite[Remark 2.3.1 (iv)]{BD2},  we see $\kappa$ is $(\varphi, \Gamma)$-equivariant (hence is a morphism of generalized $(\varphi, \Gamma)$-modules).   Since the rank of $M_{r'}$ is at most the rank of $\Delta_{r'}$, and $\Delta_{r'}$ has no $t$-torsion, we deduce using \cite[Prop. 4.12]{Berger} that  $\kappa: M_{r'}\xrightarrow{\sim} \Delta_{r'}$ (as $(\varphi,\Gamma)$-module over $\cR_E^{r'}$) and $M_{r'}$ is actually the $\cR_E^{r'}$-submodule of $\Delta_{r'}\oplus D_{r'}$ generated by $\Ima(\tilde{\mu})$ (i.e. there is no need to take closure).  Thus $\tilde{\mu}=\kappa^{-1} \circ \iota$ and $\mu=\pr_2\circ \tilde{\mu}=(\pr_2\circ \kappa^{-1}) \circ \iota$, in particular, $\mu$ factors though $\pi(\Delta)^{\vee} \otimes_E (\omega_{\Delta}\circ \dett) \xrightarrow{\iota} \Delta_{r'} \ra D_{r'}$. This concludes the proof. 
\end{proof}
\begin{remark}
	The proof of Lemma \ref{keyL} (hence of Theorem \ref{thmDel}) is crucially based on the fact that $\pi(\Delta)[u_+]|_{M(\Q_p)}$ is isomorphic, up to finite dimensional subquotients and up to twist by characters, to \textbf{two} copies of (the $E$-model of) the standard Kirillov model of generic irreducible smooth representations of $\GL_2(\Q_p)$ (i.e. $W_E$ in the proof of Theorem \ref{thmext} below). One may expect this holds in general (see \cite[Remark 2.14]{Colm18}). If so, one may deduce by the same argument that $F(\pi(D))$ is representable by $\check{D}$ for any $(\varphi, \Gamma)$-module $D$ free of rank $2$ over $\cR_E$. 
\end{remark}
 For any non-split $[D]\in \Ext^1_{(\varphi, \Gamma)}(\cR_E/t^k, t^k\Delta)$, one can associate (e.g. see \cite[Thm. 0.6 (iii)]{Colm18}) a locally analytic $\GL_2(\Q_p)$-representation $\pi(D)$ that is isomorphic to  an extension of $\pi_c(\Delta,k)$ by $\pi_{\alg}(\Delta,k)$. By \cite[Thm. 5.4.2 (ii)]{BD2}, we have:
\begin{corollary}\label{corchD}
	The functor $F(\pi(D))$ is representable by $\check{D}$. 
\end{corollary} 
\begin{proof}[Proof of Theorem \ref{thmext}]
	By Corollary \ref{corchD}, the map (\ref{eqCE}) is surjective. We prove it is injective. Let $[\pi]\in \Ext^1_{\GL_2(\Q_p)}(\pi_c(\Delta,k), \pi_{\alg}(\Delta,k))$ be non-split. Suppose $\cE([\pi])=0$, i.e. $F(\pi)$ is representable by $\check{\Delta} \oplus \cR_E(x^{1-k})/t^k$. We will use this property to construct a $\GL_2(\Q_p)$-equivariant subspace $M\subset \pi^{\vee}$ giving a splitting of $\pi^{\vee} \ra \pi_{\alg}(\Delta,k)^{\vee}$ (which leads to a contradiction). The proof is organized as follows: we first construct $M$ as an $\cR_E^+$-submodule of $\pi^{\vee}$ preserved by $\psi$ and $\Gamma$, then we show $M\neq 0$, and  $M$ is stabilized by $\GL_2(\Q_p)$ and isomorphic to $\pi_{\alg}(\Delta,k)^{\vee}$.
	
	 For $r\in \Q_{>0}$ sufficiently large, we have a natural $\cR_E^+$-linear continuous $(\psi, \Gamma)$-equivariant morphism
	$\jmath: \pi^{\vee} \ra \check{\Delta}_r \oplus \cR_E^r(x^{1-k})/t^k$ 
	such that the induced morphism
	\begin{equation}\label{jten} 
	\pi^{\vee} \otimes_{\cR_E} \cR_E^r \lra \check{\Delta}_r \oplus \cR_E^r(x^{1-k})/t^k
	\end{equation} is surjective.  Indeed, we have by \cite[Thm. 4.1.5]{BD2} a natural commutative diagram:
	\begin{equation}\label{commd}
	\begin{CD}
	0 @>>> \pi_c(\Delta,k)^{\vee}\otimes_{\cR_E^+} \cR_E^r @>>> \pi^{\vee}\otimes_{\cR_E^+} \cR_E^r @>>> \pi_{\alg}(\Delta,k)^{\vee}  \otimes_{\cR_E^+} \cR_E^r @>>> 0 \\
	@. @VVV @V (\ref{jten}) VV @VVV @. \\
	0 @>>> \check{\Delta}_r @>>>  \check{\Delta}_r \oplus \cR_E^r(x^{1-k})/t^k @>>> \cR_E^r(x^{1-k})/t^k @>>> 0
	\end{CD}.
	\end{equation}
	The left vertical map is surjective as it is induced from:
	\begin{equation*}
	\iota: \pi_c(\Delta,k)^{\vee} \cong \pi(\Delta)^{\vee} \xlongrightarrow{\iota} \check{\Delta}_r
	\end{equation*}
where the first isomorphism is $M(\Q_p)$-equivariant, and the second map is given as in the proof of Theorem \ref{thmDel}. By \cite[Lemma 3.3.5 (ii)]{BD2} and its proof, the right vertical map is also surjective, hence so is the middle vertical map. 
	Let $M:=\Ker(\pr_1 \circ \jmath: \pi^{\vee} \ra \check{\Delta}_r)$.  As the composition $\pi_c(\Delta,k)^{\vee} \hookrightarrow \pi^{\vee} \xrightarrow{\pr_1\circ \jmath} \check{\Delta}_r$ is equal to $\iota$  and hence is injective by \cite[Prop. 2.20]{Colm18}, we deduce $M\cap \pi_c(\Delta,k)^{\vee}=0$. So the following composition (continuous $\cR_E^+$-linear and $(\psi, \Gamma)$-equivariant)
	\begin{equation*}
	M \hooklongrightarrow \pi^{\vee} \twoheadrightarrow \pi_{\alg}(\Delta,k)^{\vee}
	\end{equation*}
	is injective. 
	
	\noindent (1) We first prove $M\neq 0$. Suppose $M=0$ hence $\pi^{\vee} \hookrightarrow \check{\Delta}_r$. As $\check{\Delta}_r$ is $t$-torsion free, so is $\pi^{\vee}$. From the commutative diagram (recalling $u_+=t$)
	\begin{equation*}
		\begin{CD} 0 @>>> \pi_c(\Delta,k)^{\vee} @>>> \pi^{\vee} @>>> \pi_{\alg}(\Delta,k)^{\vee} @>>> 0 \\
			@. @V u_+ VV @V u_+ VV @V u_+ VV @. \\
		0@>>>	\pi_c(\Delta,k)^{\vee} @>>> \pi^{\vee} @>>> \pi_{\alg}(\Delta,k)^{\vee} @>>> 0 
			\end{CD}
	\end{equation*}
we deduce an exact sequence (consisting of continuous maps)
\begin{equation}\label{snake}
	0 \ra \pi_{\alg}(\Delta,k)^{\vee}[u_+] \xrightarrow{\delta} \pi_c(\Delta,k)^{\vee}/u_+ \pi_c(\Delta,k)^{\vee} \ra \pi^{\vee}/u_+ \pi^{\vee} \ra \pi_{\alg}(\Delta,k)^{\vee}/u_+ \pi_{\alg}(\Delta,k)^{\vee} \ra 0.
\end{equation}
Roughly speaking, we will show a contradiction by considering the multiplicities of the Kirillov model in the dual of each term of (\ref{snake}).
By the $\GL_2(\Q_p)$-equivariant isomorphism $\pi_{\alg}(\Delta,k)^{\vee}\cong \pi_{\infty}(\Delta)^{\vee} \otimes_E (\Sym^{k-1} E^2)^{\vee}$, we get a $B(\Q_p)$-equivariant isomorphism (of reflexive Fr\'echet $E$-spaces): 
\begin{equation}\label{pialgdual}\pi_{\alg}(\Delta,k)^{\vee}[u_+]\cong \pi_{\infty}(\Delta)^{\vee} \otimes_E (1\otimes x^{1-k}).
\end{equation}  Using the isomorphisms of $B(\Q_p)$-representations 
\begin{equation*}
\pi_c(\Delta, k)\cong \pi(\Delta,1) \otimes_E (x^{-1}\otimes x^{k-1}), \ \pi(\Delta,1)[u_+]\cong \pi_{\infty}(\Delta)^{\oplus 2}, 
\end{equation*}
we deduce a 
$B(\Q_p)$-equivariant isomorphism of reflexive Fr\'echet $E$-spaces (similarly as in  the proof of Lemma \ref{keyL}, the first isomorphism following from \cite[Cor. 9.3]{DLB}):
\begin{equation}\label{picdual}
	\pi_c(\Delta,k)^{\vee}/u_+ \pi_c(\Delta,k)^{\vee}\cong \pi_c(\Delta,k)[u_+]^{\vee} \cong \big(\pi_{\infty}(\Delta)^{\vee} \otimes_E (x \otimes x^{-k})\big)^{\oplus 2}.
\end{equation} 
By similar arguments of \cite[Lemma 2.1.5]{BD2}, the injection $\delta$ induces a continuous map of spaces of compact type with dense image $\delta^{\vee}:\pi_{c}(\Delta,k)[u_+] \ra \pi_{\alg}(\Delta,k)^{\vee}[u_+]^{\vee}\cong \pi_{\infty}(\Delta) \otimes_E (1 \otimes x^{k-1})$. As $\pi_{\infty}(\Delta)$ is equipped with the finest locally convex topology, $\delta^{\vee}$ is surjective (see for example \cite[\S~5.C]{NFA}). We have hence an exact sequence of spaces of compact type (all equipped with the finest locally convex topology):
\begin{equation}\label{fvee}
	0 \ra \Ker (\delta^{\vee}) \ra \pi_{c}(\Delta,k)[u_+] \ra \pi_{\alg}(\Delta,k)^{\vee}[u_+]^{\vee} \ra 0.
\end{equation}One directly checks (by diagram chasing) that for $b\in B(\Q_p)$ and $v\in  \pi_{\alg}(\Delta,k)^{\vee}[u_+]$, $\delta(bv)=(x^{-1}\otimes x)(b)b(\delta(v))$. We see $\Ker (\delta^{\vee})$ is stabilized by $B(\Q_p)$, and the exact sequence in (\ref{fvee}) becomes $B(\Q_p)$-equivariant if we twist $\pi_{\alg}(\Delta,k)^{\vee}[u_+]^{\vee}$ by the character $x^{-1} \otimes x$ of $B(\Q_p)$.

Let $\eta: \Q_p \ra \bC_p$ be a non-trivial locally constant (additive) character. Let $W:=\cC^{\infty}_c(\Q_p^{\times},\bC_p)$ be the space of locally constant $\bC_p$-valued functions on $\Q_p^{\times}$, which is equipped with a natural $M(\Q_p)$-action given by 
\begin{equation*}
	\bigg(\begin{pmatrix}
		a & 0 \\ 0 & 1
	\end{pmatrix} f\bigg)(x)=f(ax), \ \bigg(\begin{pmatrix}
	1 & b \\ 0 & 1
\end{pmatrix} f\bigg)(x)=\eta(bx)f(x).
\end{equation*} 
Recall $W$ is irreducible and admits an $E$-model $W_E$, that is unique up to scalars in $\bC_p^{\times}$ (see \cite[Lemma 3.3.2]{BD2}). By  classical theory of Kirillov model (see for example \cite[\S~3.5]{BZI}), we have $\pi_{\infty}(\Delta)|_{M(\Q_p)}\cong W_E$. By  (\ref{pialgdual}) and using (\ref{picdual}) (\ref{fvee}), we see $\Ker(\delta^{\vee})|_{M(\Q_p)} \cong W_E \otimes_E (x^{-1} \otimes x^{1-k})$. We define $F(\Ker(\delta^{\vee}))$ exactly in the same way as for $\GL_2(\Q_p)$-representations (noting in the definition of $F(-)$, we actually only use the $M(\Q_p)$-action). By \cite[Lemma 3.3.5 (2)]{BD2},  $F(\Ker(\delta^{\vee}))$ is representable by $\cR_E(x^{-1})/t$.  

By (\ref{commd}), we have a commutative diagram
\begin{equation}
	\begin{CD} \pi_c(\Delta,k)^{\vee}/u_+ \pi_c(\Delta,k)^{\vee} @>>> \pi^{\vee}/u_+ \pi^{\vee} \\
		@V \jmath_1 VV @VVV \\
		\check{\Delta}_r/t @>>> \check{\Delta}_r/t \oplus \cR_E^r(x^{1-k})/t,
		\end{CD}
\end{equation}
such that each vertical map becomes surjective if we tensor the corresponding source by $\cR_E^r$. 
As the bottom horizontal map is obviously injective, we deduce using (\ref{snake}) that $\pi_{\alg}(\Delta,k)^{\vee}[u_+] \subset \Ker \jmath_1$ and hence $\jmath_1$ factors through (a continuous $\cR_E^+$-linear $(\psi, \Gamma)$-equivariant map) 
\begin{equation*}
	\jmath_1': \Ker(\delta^{\vee})^{\vee} \lra \check{\Delta}_r/t
\end{equation*}
which is surjective after tensoring the source by $\cR_E^r$. However, as $F(\Ker(\delta^{\vee}))$ is represented by $\cR_E(x^{-1})/t$, we see $\jmath_1'$ factors through (enlarging $r$ if needed) $\cR_E^r(x^{-1})/t \ra \check{\Delta}_r/t$, which can not be surjective, a contradiction.

\noindent (2) We show $M (\neq 0)$ is stabilized by $\GL_2(\Q_p)$ hence isomorphic to $\pi_{\alg}(\Delta,k)^{\vee}$, which will lead to a contradiction (and will conclude the proof of the theorem) as the extension $\pi$ is non-split. We begin with the following claim. 
	
	\noindent \textbf{Claim: } For $v\in \pi^{\vee}$, the followings are equivalent:
	\begin{enumerate}
		\item[(1)] $v\in M$,
		\item[(2)] $t^k v=0$,
		\item[(3)] $t^n v=0$ for $n$ sufficiently large.
	\end{enumerate}
We prove the claim. Since $M\hookrightarrow \pi_{\alg}(\Delta,k)^{\vee}$ is $\cR_E^+$-equivariant and $\pi_{\alg}(\Delta,k)^{\vee}$ is annihilated by $t^k$, we see (1) $\Rightarrow$ (2). (2) $\Rightarrow$ (3) is trivial. Suppose $t^n v=0$ for some $n$, then $\pr_1 \circ \jmath(t^n v)=t^n \pr_1 \circ \jmath(v)=0$. Since $\check{\Delta}_r$ has no $t$-torsion,  we see $\pr_1\circ \jmath(v)=0$, i.e. $v\in M$. 

For $v\in \pi^{\vee}$, $b\in B(\Q_p)$, we have $t^n (bv)=(u_+)^n\cdot (b v)=b (\Ad_{b^{-1}} (u_+)^n \cdot v)$. If $t^n v=0$, then $\Ad_{b^{-1}} (u_+)^n \cdot v=0$ thus $t^n (bv)=0$. By the claim, we see $M \subset \pi_{\alg}(\Delta,k)^{\vee}$ is stabilized by $B(\Q_p)$. 

Next we show $M$ is stabilized by $u_-=\begin{pmatrix}
0 & 0 \\ 1 & 0
\end{pmatrix}\in \gl_2$. Since $M$ is a $B(\Q_p)$-submodule of $\pi_{\alg}(\Delta,k)^{\vee}$,  we see for any $v\in M$, the $\ub$-module generated by $v$ is finite dimensional and is spanned by eigenvectors of $h=\begin{pmatrix}
1 & 0 \\ 0 & -1
\end{pmatrix}\in \gl_2$.  Using the relation $[(u_+)^{n+1},u_-]=n (u_+)^n(h+n)$ in $\text{U}(\gl_2)$, we deduce  for $v\in M$, $t^n(u_-\cdot v)=(u_+)^n u_- \cdot v$=0 for $n$ sufficient large and hence $u_- \cdot v\in M$ by the claim. Consequently, $M$ is a $\text{U}(\gl_2)$-submodule of $\pi^{\vee}$ and the injection $M\hookrightarrow \pi_{\alg}(\Delta,k)^{\vee}$ is $\text{U}(\gl_2)$-equivariant. We deduce then any vector $v$ in $M$ is annihilated by $(u_-)^k$.

Let $w=\begin{pmatrix} 0 & 1 \\ 1 & 0 \end{pmatrix}\in \GL_2(\Q_p)$. For $v\in M$, we have $t^k w v=w (\Ad_w(u_+)^k \cdot v)=w (u_-^k \cdot v)=0$.  Thus $M$ is stabilized by $w$, hence is stabilized by $\GL_2(\Q_p)$ (recalling $M$ is $B(\Q_p)$-invariant). Since $\pi_{\alg}(\Delta,k)$ is irreducible, we deduce $M\cong \pi_{\alg}(\Delta,k)^{\vee}$. As previously discussed, this finishes the proof. 
\end{proof}
\begin{proof}[Proof of Corollary \ref{corext}] The ``if" part is a trivial consequence of Theorem \ref{thmext}. Assume now	$$\Ext^1_{\GL_2(\Q_p)}(\pi_c(\Delta,k), \pi_{\infty} \otimes_E W)\neq 0.$$ Then $\pi_{\infty} \otimes_E W$ has the same central character and infinitesimal character as 
$\pi_c(\Delta,k)$. By \cite[Prop. 3.1.1]{Colm18}, one  deduces $W\cong \Sym^{k-1} E^2$. 
	
Similarly as in Corollary \ref{corcE} (using Corollary \ref{corchD}, \cite[Thm. 3.3.1 \& Thm. 4.1.5]{BD2}), we have a morphism
\begin{equation*}\cE': 
\Ext^1_{\GL_2(\Q_p)}\big(\pi_c(\Delta,k), \pi_{\infty} \otimes_E \Sym^{k-1} E^2\big) \lra \Ext^1_{(\varphi, \Gamma)}(\cR_E(x^{1-k})/t^k, \check{\Delta}).
\end{equation*}
By the same argument as in the proof of Theorem \ref{thmext} (with $\pi_{\alg}(\Delta,k)$ replaced by $\pi_{\infty} \otimes_E \Sym^{k-1} E^2$), the morphism is injective. Suppose $\pi_{\infty}$ is not isomorphic to $\pi_{\infty}(\Delta)$ and there exists a non-split $[\pi]\in \Ext^1_{\GL_2(\Q_p)}\big(\pi_c(\Delta,k), \pi_{\infty} \otimes_E \Sym^{k-1} E^2\big)$. Let $[\check{D}]:=\cE'([\pi])$, and let $[\pi(D)]:=\cE^{-1}([\check{D}])$. The pull-back of $\pi_c(\Delta,k)$ of $\pi(D) \oplus \pi\twoheadrightarrow \pi_c(\Delta,k)^{\oplus 2}$ via the diagonal map gives a non-split extension $\tilde{\pi}$ of $\pi_c(\Delta,k)$ by $(\pi_{\infty} \otimes_E \Sym^{k-1} E^2) \oplus \pi_{\alg}(\Delta,k)$ satisfying $\tilde{\pi}/\pi_{\alg}(\Delta,k)\cong \pi$ and $\tilde{\pi}/(\pi_{\infty} \otimes_E \Sym^{k-1} E^2)\cong \pi(D)$. By \cite[Thm. 4.1.5]{BD2}, $F(\tilde{\pi})$ is representable by an extension of $(\cR_E(x^{1-k})/t^k)^{\oplus 2}$ by $\check{\Delta}$ such that the pull-back of either of the two factors $\cR_E(x^{1-k})/t^k$ is isomorphic to $\check{D}$. We deduce then $F(\tilde{\pi})$ is representable by $\check{D}\oplus \cR_E(x^{1-k})/t^k$.
We have thus a continuous $\cR_E^+$-linear $(\psi, \Gamma)$-equivariant morphism when $r$ is sufficiently large:
\begin{equation*}
\jmath: \tilde{\pi}^{\vee} \lra \check{D}_r\oplus \cR_E^r(x^{1-k})/t^k
\end{equation*}
such that the morphism becomes surjective if we tensor the source by $\cR_E^r$ (by similar arguments as for the surjectivity of (\ref{jten})). Let $M$ be the kernel of $\pr_1 \circ \jmath$. Since $F(\pi(D))(\check{D})\cong \End_{(\varphi, \Gamma)}(\check{D})\cong E$,  the restriction of $\pr_1 \circ \jmath$ on $\pi(D)^{\vee}$ is equal, up to non-zero scalars, to the morphism $\pi(D)^{\vee} \ra \check{D}$ in \cite[Prop. 2.20]{Colm18} hence is injective. Using a similar  exact sequence as in (\ref{snake}) with $\pi$ replaced by $\pi(D)$, we can deduce $\pi(D)[u_+]|_{M(\Q_p)}\cong (W_E \otimes_E (x^{-1} \otimes x^{k-1}))^{\oplus 2}$ (see also \cite[Remark 3.3.2]{Colm18}). Now by the same arguments as in the proof of Theorem \ref{thmext} \big(with $\pi_c(\Delta,k)$ replaced by $\pi(D)$ and $\pi_{\alg}(\Delta,k)$ replaced by $\pi_{\infty} \otimes_E \Sym^{k-1} E^2$\big), one can prove $M$ is $\GL_2(\Q_p)$-invariant, and is isomorphic to $(\pi_{\infty} \otimes_E \Sym^{k-1} E^2)^{\vee}$. Hence $\tilde{\pi}\cong \pi(D)\oplus \pi_{\infty} \otimes_E \Sym^{k-1} E^2$ and then $\pi\cong \tilde{\pi}/\pi_{\alg}(\Delta,k)\cong \pi_c(\Delta,k) \oplus \pi_{\infty} \otimes_E \Sym^{k-1} E^2$ \big(noting $\Hom_{\GL_2(\Q_p)}\big(\pi_{\alg}(\Delta,k), \pi_{\infty} \otimes_E \Sym^{k-1} E^2\big)=0$ by assumption\big), a contradiction.
\end{proof}
\subsection*{Acknowledgement} This note is motivated by my joint work \cite{BD2} with Christophe Breuil, and I also thank him for correspondences on the problem. I thank the anonymous referee for the rapid and  sharp comments. The work is supported by the NSFC Grant No. 8200905010 and No. 8200800065. 


\begin{thebibliography}{10}
	
	\bibitem{Berger}
	Laurent Berger.
	\newblock Repr{\'e}sentations $p$-adiques et {\'e}quations diff{\'e}rentielles.
	\newblock {\em Inventiones mathematicae}, 148(2):219--284, 2002.
	
	\bibitem{Ber08a}
	Laurent Berger.
	\newblock {\'E}quations diff{\'e}rentielles $p$-adiques et $(\varphi,
	{N})$-modules filtr{\'e}s.
	\newblock {\em Ast{\'e}risque}, 319:13--38, 2008.
	
	\bibitem{BZI}
	I.~N. Bernstein and A.~V. Zelevinsky.
	\newblock Induced representations of reductive $p$-adic groups. {I}.
	\newblock {\em Annales scientifiques de l'{\'E}cole {N}ormale
		{S}up{\'e}rieure}, 10(4):441--472, 1977.
	
	\bibitem{Br16}
	Christophe Breuil.
	\newblock $\mathrm{Ext}^1$ localement analytique et compatibilit\'e
	local-global.
	\newblock {\em American J. of Math},
	 141: 611--703, 2019. 
	
	\bibitem{BD2}
	Christophe Breuil and Yiwen Ding.
	\newblock Sur un probl\`eme de compatibilit\'e local-global localement
	analytique.
	\newblock {\em Memoirs of the Amer. Math. Soc.}
	\newblock to appear.
	
	\bibitem{BS07}
	Christophe Breuil and Peter Schneider.
	\newblock First steps towards p-adic {Langlands} functoriality.
	\newblock {\em Journal f{\"u}r die reine und angewandte Mathematik (Crelles
		Journal)}, 2007(610):149--180, 2007.
	
	\bibitem{Colm10a}
	Pierre Colmez.
	\newblock Repr{\'e}sentations de $\mathrm{GL}_2(\mathbb{Q}_p)$ et
	{$(\varphi,\Gamma)$}-modules.
	\newblock {\em Ast{\'e}risque}, 330:281--509, 2010.
	
	\bibitem{Colm16}
	Pierre Colmez.
	\newblock Repr\'esentations localement analytiques de
	{$\mathrm{GL}_2(\mathbb{Q}_p)$} et $(\varphi, {\Gamma})$-modules.
	\newblock {\em Representation theory}, 20:187--248, 2016.
	
	\bibitem{Colm18}
	Pierre Colmez.
	\newblock Correspondance de {Langlands} locale $ p $-adique et changement de
	poids.
	\newblock {\em Journal of the European Mathematical Society}, 21(3):797--838,
	2018.
	
	\bibitem{CD}
	Pierre Colmez and Gabriel Dospinescu.
	\newblock Compl\'et\'es universels de repr\'esentations de
	{$\mathrm{GL}_2(\mathbb{Q}_p)$}.
	\newblock {\em Algebra and number theory}, 8(6):1447--1519, 2014.
	
	\bibitem{DLB}
	Gabriel Dospinescu and Arthur-C\'esar Le~Bras.
	\newblock Rev\^etements du demi-plan de {D}rinfeld et correspondance de
	{L}anglands $p$-adique.
	\newblock {\em Annals of mathematics}, 186(2): 321--411, 2017.
	
	\bibitem{Em4}
	Matthew Emerton.
	\newblock Local-global compatibility in the $p$-adic {Langlands} programme for
	$\mathrm{GL}_2/\mathbb{Q}$.
	\newblock 2011.
	\newblock preprint.
	
	\bibitem{HT}
	Michael Harris and Richard Taylor.
	\newblock The geometry and cohomology of some simple {Shimura} varieties.
	\newblock {\em Annals of Math. Studies}, 151, 2001.
	
	\bibitem{Liu07}
	Ruochuan Liu.
	\newblock Cohomology and duality for ({$\varphi$}, {$\Gamma$})-modules over the
	{Robba} ring.
	\newblock {\em International Mathematics Research Notices}, Vol. 2007,  32 pages, 2007.
	
	\bibitem{NFA}
	Peter Schneider.
	\newblock Nonarchimedean functional analysis.
	\newblock{\em Springer Monographs in Mathematics}, 2002.
\end{thebibliography}
\end{document}